\documentclass[11pt]{article}
\usepackage[margin=1in]{geometry}
\usepackage{amsmath,amssymb,amsthm,mathtools,booktabs,microtype}
\usepackage{xurl}
\usepackage[hidelinks]{hyperref}
\newtheorem{theorem}{Theorem}[section]
\newtheorem{lemma}[theorem]{Lemma}
\newtheorem{proposition}[theorem]{Proposition}
\newtheorem{corollary}[theorem]{Corollary}
\theoremstyle{definition}

\theoremstyle{remark}
\newtheorem*{remark}{Remark}
\numberwithin{equation}{section}
\newcommand{\Z}{\mathbb Z}
\newcommand{\R}{\mathbb R}
\newcommand{\F}{\mathbb F}
\newcommand{\ones}{\mathbf 1}
\DeclareMathOperator{\conv}{conv}
\DeclareMathOperator{\wt}{wt}

\title{Integrality, smoothness and normality bounds\\for cube-truncated Hadamard simplices}
\author{Nikita Lebedev}
\date{September 2026}
\hypersetup{
 pdftitle={Integrality, smoothness and normality bounds for cube-truncated Hadamard simplices},
 pdfauthor={Nikita Lebedev},
 pdfsubject={Lattice polytopes and integer decomposition},
 pdfkeywords={Hadamard simplices, lattice polytopes, normality, smoothness}
}
\begin{document}
\maketitle
\begin{abstract}
Santos asked when intersections of dilated Hadamard simplices with cubes
are integral, smooth, or normal, in a prescribed affine lattice. We
construct a nonintegral example in dimension eleven and prove that no
smaller-dimensional example exists. We characterize smoothness completely
and show that every smooth member of this family is normal. An explicit
example in dimension fifteen shows that integrality alone does not imply
normality.

For Sylvester simplices of order at least sixteen, we establish a sharp
uniform integrality bound and construct counterexamples immediately below
it. We also obtain sufficient normality bounds for general Hadamard
simplices and stronger bounds for the Sylvester family. The proofs use
integer decomposition for boxes with separated corner cuts and rounding
under three signed slab constraints. All numbered results have formal
counterparts verified in Lean.
\end{abstract}

\section{Introduction}\label{intro:results}
Oda's normality conjecture asks whether every smooth lattice polytope
is normal. Santos studied cube-truncated Hadamard simplices as possible
counterexamples and asked whether they are always lattice polytopes,
and which parameters make them normal or smooth~\cite[Question~5]{Santos}.
We retain his specified ambient lattice throughout.

\subsection{The family and its lattice}\label{common:family}
Mark a column of a real Hadamard matrix $H$ of order $N$, normalize it
to ones by row signs, and delete it. For $d=N-1$, the remaining rows
$v_0,\ldots,v_{N-1}\in\{\pm1\}^d$ satisfy
\begin{equation}\label{common:gram}
 v_i\cdot v_j=N\delta_{ij}-1,\qquad \sum_i v_i=0.
\end{equation}
They form the simplex $\Delta=\conv\{v_i\}$, with facet inequalities
$v_i\cdot x\ge-1$. For positive odd integers $n<m<dn$, set
\begin{equation}\label{common:source}
 P_{m,n}=m\Delta\cap[-n,n]^d,\qquad \Lambda_n=(n+2\Z)^d.
\end{equation}
Fix the marked column. Signed row/column (Hadamard) equivalence allows
permutations and sign changes of rows and columns. We call $x$ the
source coordinates and use the integer chart
\begin{equation}\label{common:chart}
 Q_{m,n}=\frac{P_{m,n}+n\ones}{2}
 =\{y\in[0,n]^d:v_i\cdot y\ge b_i\ \forall i\},\qquad
 b_i=\frac{n\sum_j(v_i)_j-m}{2}.
\end{equation}
Normality means the ambient integer decomposition property (IDP): for every
integer $k\ge1$, every $Z\in kQ_{m,n}\cap\Z^d$ is a sum of $k$ points of
$Q_{m,n}\cap\Z^d$. At degree $k$, the source lattice is
$(kn+2\Z)^d$ and the chart is $z\mapsto(z+kn\ones)/2$.
We do not replace the ambient lattice by the subgroup generated by the
vertices. A lattice polytope is smooth if the primitive facet normals
at each vertex form a basis of the dual lattice, identified with $\Z^d$
in the integer chart.

\subsection{Results and approach}
We answer Santos's Question~5(1) negatively: the counterexample in
Proposition~\ref{base:paley-counter} is nonintegral, and dimension eleven
is minimal (Corollary~\ref{base:min-dimension}). We answer Question~5(3)
completely: smoothness holds exactly when
$m>(d-1)n/2$ and $m/n\notin2\Z+1$
(Theorem~\ref{base:smoothness}). Every smooth member is normal
(Corollary~\ref{common:upper-normality}), whereas integrality alone does
not imply normality (Proposition~\ref{common:nonnormal-example}).
For Sylvester orders $N\ge16$, Theorem~\ref{syl:all-orders} gives the
sharp uniform integrality band $m\ge(3N/8-1)n$.
Question~5(2), concerning normality, remains partly open: that theorem
and Theorem~\ref{syl:generic} give sufficient normality bands.

The decomposition proofs use separated corner cuts
(Theorem~\ref{common:strict-cuts}) and rounding under at most three
signed slabs (Lemmas~\ref{common:slabs}--\ref{common:rounding}).

\subsection{Relation to earlier results}
Bruns studies chiseling of smooth polytopes~\cite[Section~6]{Bruns}.
His Theorem~6.2 deduces integral closedness of a union from that of
both pieces of a hyperplane cut; it does not infer IDP for a retained
piece from that of the original polytope. Theorem~\ref{common:strict-cuts}
proves IDP directly for a box with strictly separated corner cuts.
We use Orrick's terminology for closed Hadamard quadruples~\cite[Section~3]{Orrick}.
Curtis proves IDP for smooth combinatorial cubes~\cite[Theorem~4.8]{Curtis};
our truncations have $3d+1$ facets
(proof of Theorem~\ref{base:smoothness}), so this result does not directly apply.
Saraf and Varadarajan characterize the integer points of the undilated
Sylvester simplex and give lower bounds on the number of integer points
in its dilates~\cite[Theorems~2 and~4]{SV}.
Here the conclusion is a decomposition of
\emph{every} lattice target in a dilation of the cube intersection.

Results on large integer dilates of a fixed integral polytope,
such as Cox--Haase--Hibi--Higashitani's Theorem~3.2~\cite{CHHH},
also have different hypotheses: varying $m$ with $n$ fixed does not
uniformly dilate the intersection, and integrality must first be proved.
We do not use Santos's lower-region nonnormality assertion for
$m<N/4$~\cite[Corollary~27 and p.~2308]{Santos}.
Our sufficient bands do not classify all parameters or resolve Oda's
conjecture for arbitrary smooth lattice polytopes.

\section{Integrality and the minimum counterexample dimension}\label{base:section}\subsection{A counterexample in the least possible dimension}
\begin{proposition}\label{base:paley-counter}
There is a Hadamard simplex in dimension $11$ for which $P_{m,n}$ is nonintegral whenever $n<m<7n/3$ and $m-n\equiv2\pmod4$. In particular, $P_{5,3}$ is nonintegral.
\end{proposition}
\begin{proof}
Let $\chi$ be the quadratic character modulo $11$. Form the order-$12$ Paley matrix with first row and column all ones and remaining block
\[
 B_{ij}=\begin{cases}-1&i=j,\\\chi(j-i)&i\ne j,\end{cases}
 \qquad i,j\in\{0,\ldots,10\}.
\]
Delete its initial constant column, obtaining normals $v_0,\ldots,v_{11}$ in the indicated order. Direct multiplication verifies~\eqref{common:gram}.
Put $a=(m-n)/2$ and consider
\[
 x=(-n,-n,-n,n-3a,a-n,a-n,n-a,a-n,n,n-a,n).
\]
For $0\le a\le2n/3$ this point is in the cube. Its twelve simplex scalar products are
\begin{align*}
 (&-2a-n,-2a-n,4a-n,-2a-n,4a-n,4a-n,\\
  &11n-8a,4a-n,4a-n,-2a-n,-2a-n,-2a-n).
\end{align*}
All are at least $-m=-n-2a$. Coordinates $1,2,3,9,11$ are fixed cube coordinates. On the remaining coordinates $4,5,6,7,8,10$, tight simplex rows $0,1,3,9,10,11$ form the matrix
\[
 \begin{pmatrix}
 1&1&1&1&1&1\\
 1&1&1&-1&-1&1\\
 1&-1&1&1&1&-1\\
 -1&-1&-1&1&-1&1\\
 1&-1&-1&-1&1&-1\\
 1&1&-1&-1&-1&-1
 \end{pmatrix},
 \qquad \det=64.
\]
Thus $x$ is a vertex. If $a$ is odd, its six free coordinates are even, whereas $\Lambda_n$ requires odd coordinates. At $n=3,m=5$, this gives the explicit vertex
\[
 (-3,-3,-3,0,-2,-2,2,-2,3,2,3).
\]
\end{proof}
\begin{remark}[The counterexample in the integer chart]
At $(m,n)=(5,3)$ the displayed vertex has integer coordinates, but
six of them lie outside the odd source coset. In the chart
$y=(x+3\ones)/2$ it becomes
\[
 y=\left(0,0,0,\frac32,\frac12,\frac12,
              \frac52,\frac12,3,\frac52,3\right).
\]
Its six half-integer coordinates make nonintegrality in the required
lattice explicit.
\end{remark}
\subsection{A universal integrality band}
\label{base:integrality-section}

The integrality range extends below the smoothness threshold. The endpoint
in the following theorem is included. We use the elementary fact that a
Hadamard order $N>2$ is divisible by four: after normalizing one row,
orthogonality forces the four sign pairs in two other rows to occur
$N/4$ times each.

\begin{lemma}\label{base:small-sign-basis}
Let $M$ be an invertible $k\times k$ matrix with entries $\pm1$, where
$1\le k\le4$. If $|\det M|=2^{k-1}$, then $M^{-1}q\in\Z^k$ for every vector
$q$ with odd integer entries, and $2M^{-1}$ is integral. Every invertible
such matrix with $k\le3$ has this determinant magnitude. For $k=4$, the only
other possibility is $|\det M|=16$, in which case $MM^{\mathsf T}=4I$.
\end{lemma}
\begin{proof}
Normalize the first row and column to ones by row and column sign changes.
Subtracting the first row from the other rows shows that $2^{k-1}$ divides
the determinant. If its magnitude is $2^{k-1}$, the remaining
$(k-1)\times(k-1)$ block, divided by two, is a unimodular $0,-1$ matrix.
Applying these operations to $Mz=q$ gives integral right-hand sides
$(q_i-q_1)/2$, so $z$ is integral. Differences of sign vectors include
$2e_i$, proving that $2M^{-1}$ is integral. The case $k=1$ is immediate.
The determinant bound $|\det M|\le k^{k/2}$ gives the listed possibilities.
Equality at $k=4$ forces the rows to be orthogonal.
\end{proof}

\begin{theorem}\label{base:integrality}
Let $\Delta$ be any Hadamard simplex of dimension $d\ge15$, and let $n,m$
be positive odd integers with $n<m<dn$. Then
\[
 m\ge\frac{d-5}{2}n
 \quad\Longrightarrow\quad
 P_{m,n}=m\Delta\cap[-n,n]^d
 \text{ is a lattice polytope in }\Lambda_n=(n+2\Z)^d.
\]
\end{theorem}
\begin{proof}
Fix a vertex $x$. Include all active cube normals in a basis of active
constraints, and complete it with $k$ simplex normals. Their restriction
to the $k$ free coordinates $F$ is an invertible sign matrix $M$;
$-n<x_j<n$ on $F$, and all other coordinates are at cube bounds.
If $k=0$, $x\in\Lambda_n$ already.

We first show $k\le4$. For five distinct selected normals let
$s=v_1+\cdots+v_5$. The Gram equations give
$\sum_j s_j^2=5(d-4)$. Each $s_j$ is odd, so
$4|s_j|\le s_j^2+3$ and $\|s\|_1\le2d-5$.
Their tight equations imply
\begin{equation}\label{base:five-facet-bound}
 5m=-s\cdot x\le n\|s\|_1\le(2d-5)n.
\end{equation}
For $d>15$ this contradicts the assumed band. For $d=15$,
equality is necessary. Since every $s_j$ is odd and nonzero, equality
in the cube bound forces every coordinate to a cube endpoint,
contradicting $k\ge5$. Hence $k\le4$, including at the endpoint.

It remains to exclude the exceptional case $k=4$, $|\det M|=16$.
Let $A$ consist of the four selected full rows, and put
$s_j=\sum_i A_{ij}$. Then
\[
 AA^{\mathsf T}=(d+1)I-J,\qquad MM^{\mathsf T}=4I,\qquad
 \sum_j s_j^2=4(d-3),\quad \sum_{j\in F}s_j^2=16.
\]
Here $J$ is the $4\times4$ all-ones matrix and
$s_j\in\{0,\pm2,\pm4\}$. If any $|s_j|=4$, the inequality
$2|s_j|\le s_j^2$ improves by eight at that coordinate and gives
$\|s\|_1\le2(d-5)$. Thus
$4m=-s\cdot x\le2(d-5)n$. Equality is again necessary, forcing
$s_j=0$ on $F$, contrary to its square sum of 16.

Coordinate sign changes preserve the cube, $\Lambda_n$, the free
coordinates and $|\det M|$. Make such
changes so that every nonzero column sum is two.
Let $U$ and $Z$ index the columns of sum two and zero, respectively.
The square-sum identities give $|U|=d-3$, $|Z|=3$, and $F\subset U$.
The Gram identity also gives, for each row $i$,
\[
 2\sum_{j\in U}A_{ij}=d-3,\qquad
 \sum_{j\in U}(1-A_{ij})=(d-3)/2.
\]
Since the coordinates indexed by $Z$ are fixed, write $x_j=n\eta_j$
there and set $z_i=\sum_{j\in Z}A_{ij}\eta_j$. The four integers
$z_i$ are odd and sum to zero; choose $i$ with $z_i\le-1$.
This row has a negative entry in some free column. Otherwise its
restriction to $F$ would be $\ones$, and the sum of all four restricted
rows, which is $2\ones$, would contradict their independence.

The nonnegative coefficients $1-A_{ij}$ therefore include a positive
coefficient at a free coordinate. The strict cube bound gives
\[
 -\frac{d-3}{2}n
 <\sum_{j\in U}(1-A_{ij})x_j
 =-m+nz_i
 \le-m-n
 \le-\frac{d-3}{2}n,
\]
a contradiction. The equality follows by summing the four tight row
equations to get $\sum_{j\in U}x_j=-2m$, and using
$\sum_{j\in U}A_{ij}x_j=-m-nz_i$.

Finally normalize any full simplex row to $\ones$ by coordinate signs.
The row sums are $d$ and $-1$. Since the Hadamard order $d+1$ is
divisible by four, the right-hand sides in~\eqref{common:chart},
$b_i=(n\sum_jv_i(j)-m)/2$, share one parity.
The fixed chart coordinates are $0$ or $n$; subtracting their signed
contributions preserves common parity in the restricted system.
The exclusion above and Lemma~\ref{base:small-sign-basis} give
$|\det M|=2^{k-1}$. That lemma makes the solution integral: write any
common-parity vector as $q=c\ones+2w$, with $c\in\Z$, $w\in\Z^k$;
then both $M^{-1}\ones$ and $2M^{-1}w$ are integral.
All chart coordinates are therefore integers, so every vertex lies
in $\Lambda_n$.
\end{proof}

\begin{proposition}\label{base:d7-integral}
Every Hadamard truncation in dimension seven is a lattice polytope in
$\Lambda_n$ for positive odd $n<m<7n$.
\end{proposition}
\begin{proof}
Any three distinct Hadamard vertices $u,v,w$ agree in exactly
\[
 \frac14\sum_{j=1}^7(1+u_jv_j+u_jw_j+v_jw_j)
 =\frac{7-3}{4}=1
\]
coordinate. Their scaled triangle in $m\Delta$ therefore has that
coordinate identically $m$ or $-m$, and cannot meet $[-n,n]^7$ when
$m>n$.

For $x\in m\Delta$, its barycentric coordinates are
$\lambda_i=(v_i\cdot x+m)/(8m)$. Five tight simplex inequalities
would leave support on at most three vertices. Extending a smaller
support to a triple if necessary, the preceding observation excludes
this at every feasible point of $P_{m,n}$. Thus at most four simplex
inequalities are active.

Include all active cube normals in a vertex basis as in the proof of
Theorem~\ref{base:integrality}. The remainder of that proof uses
only the bound $k\le4$ and $m\ge(d-5)n/2$; its dimension restriction
was used solely to obtain the active-facet bound. Here $d=7$ and
$m>n=(d-5)n/2$, so the determinant-$16$ exclusion and the final
sign-matrix parity argument apply. Every vertex is in $\Lambda_n$.
\end{proof}

\begin{corollary}\label{base:min-dimension}
Eleven is the least dimension admitting a nonintegral member under
the source assumptions.
\end{corollary}
\begin{proof}
For $d>1$, the Hadamard order $d+1$ is divisible by four
(Section~\ref{base:integrality-section}).
Dimensions zero and one have no admissible parameters. The only possible smaller
dimensions are therefore three and seven.

In dimension three, two tight simplex normals $v,w$ would give
$-m=(v+w)\cdot x/2\ge-n$, contrary to $m>n$. Each vertex is thus a
cube vertex or a cut of a cube edge by one simplex inequality, and its
coordinates are odd. Proposition~\ref{base:d7-integral} covers dimension
seven. Proposition~\ref{base:paley-counter} supplies the dimension-eleven
counterexample.
\end{proof}

\subsection{The complete smoothness criterion}
\begin{theorem}\label{base:smoothness}
For any Hadamard simplex and positive odd $n<m<dn$, the polytope $P_{m,n}$ is smooth in $\Lambda_n$ if and only if
\begin{equation}\label{base:smoothness-equation}
 m>\frac{d-1}{2}n\qquad\text{and}\qquad m/n\notin2\Z+1.
\end{equation}
Throughout the strict upper regime $m>(d-1)n/2$, it is a lattice polytope, including at the excluded resonance ratios.
\end{theorem}
\begin{proof}
All displayed inequalities are genuine facets. The point $-(m/d)v_r$ is in the relative interior of simplex facet $r$ and strictly inside the cube and other simplex inequalities. The point $\pm ne_j$ similarly witnesses each cube facet, since $m>n$.

Distinct Hadamard vertices $v,w$ agree in exactly $(d-1)/2$ coordinates. If both corresponding facets are tight at a point of $[-n,n]^d$, then
\[
 -m=\tfrac12(v+w)\cdot x\ge-\tfrac{d-1}{2}n.
\]
Thus no two simplex facets meet in the upper regime. Every vertex is consequently a cube vertex or the intersection of one simplex facet with a cube edge. In the edge case the last free coordinate is $\pm m$ plus a signed sum of $d-1$ copies of $n$, so it is odd. This proves lattice integrality.

A simplex facet passes through a cube vertex only when $m/n$ is an odd integer. Away from those ratios, a cube vertex has its $d$ cube facets, and an edge-cut vertex has $d-1$ cube facets and one simplex facet. In the chart $x=2y-n\ones$, the primitive normals are signed coordinate vectors and $\pm v_r$. Each such matrix has determinant $\pm1$, proving smoothness.

At a resonance $m/n=d-2j$ in the upper regime, one has $1\le j<(d+1)/4$. By coordinate sign changes, take one normal to be $\ones$, and take a cube vertex with $j$ positive and $d-j$ negative coordinates. It lies on that simplex facet. Every other simplex inequality is strict: a tight or violated second inequality would contradict the preceding averaging bound. In nonnegative coordinates $u$ for inward cube directions, its tangent cone is
\[
 u_i\ge0,\qquad
 \sum_{i:\,x_i=-n}u_i-\sum_{i:\,x_i=n}u_i\ge0.
\]
All $d+1$ inequalities are irredundant. For a coordinate equality, make the other coordinates positive with sufficiently large total on the negative-coordinate side; this remains possible after deleting one such coordinate because $d-j\ge2$. For the last equality, balance positive totals on both sides. The vertex is not simple.

Finally suppose $m\le(d-1)n/2$. For two distinct normals $v,w$, the point
\[
 z=-\frac{m}{d-1}(v+w)
\]
lies in the cube and on both simplex facets. Every other normal $u$ satisfies $u\cdot z=2m/(d-1)>-m$. Their nonempty bounded intersection face contains a vertex. If it has more than $d$ incident facets it is not simple. Otherwise its primitive-normal matrix contains two $\pm1$ rows, which are identical modulo~$2$; its determinant cannot be $\pm1$. Smoothness is impossible.
For $d=3$, the lower regime is absent and no odd integer lies strictly between $1$ and $3$. For $d=1$ there are no admissible parameters.
\end{proof}

\begin{remark}
Santos states the sufficient smoothness bound $m>(3d-1)n/4$
in the original report~\cite[p.~2308]{Santos}. That statement needs
a resonance qualification, supplied by Theorem~\ref{base:smoothness}. For every Hadamard dimension $d>7$, the admissible choice $m=(d-2)n$ satisfies that strict bound but is nonsmooth by the theorem. For example $d=11,m=9,n=1$ has such a vertex.
\end{remark}

\section{Corner cuts and normality}\label{common:corners}
\begin{theorem}[Strictly separated corner cuts]\label{common:strict-cuts}
Let $B=\prod_{j=1}^d[0,\ell_j]$ have positive integer side lengths. Let $\mathcal C$ be a finite set of distinct corners, with nonnegative integer thresholds $t_c$. Suppose
\[
 \|c-c'\|_1>t_c+t_{c'}\quad(c\ne c').
\]
If the polytope
\[
 Q=\{x\in B:L_c(x)\ge t_c\text{ for all }c\in\mathcal C\},
 \qquad L_c(x)=\sum_j|x_j-c_j|,
\]
is nonempty, it is integral and has IDP.
\end{theorem}
\begin{proof}
Each $L_c$ is affine on $B$. The triangle inequality implies that at any point of $B$ at most one cut is tight or violated. Therefore every vertex of $Q$ is a box vertex or an edge intersection with one cut. Since the free-coordinate coefficient is $\pm1$, these vertices are integral.

Fix an integer $k\ge1$ and $z\in kQ\cap\Z^d$. Split each coordinate of $z$ into $k$ integers between $0$ and its box bound, obtaining integer box points $p_1,\ldots,p_k$ summing to $z$. Affinity gives
\[
 \sum_iL_c(p_i)=kL_c(z/k)\ge kt_c.
\]
Consider the nonnegative integer potential
\[
 \Phi=\sum_i\sum_{c\in\mathcal C}\max\{0,t_c-L_c(p_i)\}.
\]
If it is positive, choose a deficient point $a$ for a corner $c$. The averaging inequality supplies another point $b$ with $L_c(b)>t_c$.

If every cut at $b$ is strictly satisfied, choose a coordinate $j$ with
$|b_j-c_j|>|a_j-c_j|$, which exists because $L_c(b)>L_c(a)$.
Otherwise let $c'\ne c$ be the unique tight or violated cut at $b$, and put $J=\{j:c_j\ne c'_j\}$. Opposite endpoint distances yield
\begin{align*}
 \sum_{j\in J}\bigl(|b_j-c_j|-|a_j-c_j|\bigr)
 &\ge\|c-c'\|_1-L_{c'}(b)-L_c(a)\\
 &>\|c-c'\|_1-t_{c'}-t_c>0.
\end{align*}
Choose such a coordinate $j\in J$ with positive summand.

Transfer one unit from $b$ to $a$ away from $c$ in coordinate $j$. The strict integer coordinate-distance inequality ensures both points stay in the box. Their sum is preserved. The deficit at $a$ decreases by one. Every other cut at $a$ had integer slack at least one, so none becomes violated. In the first case the same holds for every cut at $b$. In the second case the move takes $b$ away from $c'$, increasing its distance, while every other cut at $b$ had slack at least one. Thus $\Phi$ decreases strictly. Repetition terminates with all summands in $Q$, proving IDP.
\end{proof}

\begin{corollary}\label{common:upper-normality}
Every admissible Hadamard truncation with $m>(d-1)n/2$ is normal. In particular, every smooth member of the source family is normal.
\end{corollary}
\begin{proof}
In the chart $x=2y-n\ones$, the box is $[0,n]^d$. To normal $v$ associate the corner $c_v=(n/2)(\ones-v)$. The simplex inequality becomes
\[
 L_{c_v}(y)\ge t=\frac{dn-m}{2}.
\]
The threshold is integral, and distinct corners have distance $(d+1)n/2$. The strict-separation condition is precisely $m>(d-1)n/2$. Apply Theorem~\ref{common:strict-cuts}. For $X\in kP_{m,n}\cap(kn+2\Z)^d$, put $Y=(X+kn\ones)/2$. A decomposition $Y=\sum_i y_i$ into integer points of the transformed polytope gives $x_i=2y_i-n\ones\in P_{m,n}\cap\Lambda_n$ and $X=\sum_i x_i$. Theorem~\ref{base:smoothness} gives the last assertion.
\end{proof}

\begin{proposition}[An integral nonnormal member]\label{common:nonnormal-example}
For the Sylvester simplex of dimension fifteen, $P_{5,1}$ is integral
in $(1+2\Z)^{15}$ and is not normal.
\end{proposition}
\begin{proof}
Integrality is the included endpoint of Theorem~\ref{base:integrality}.
Index rows by $0,\ldots,15$ and coordinates by $1,\ldots,15$,
identifying each index with its four-bit vector in $\F_2^4$.
The entries are $(-1)^{r\cdot j}$, with the dot product taken over $\F_2$. Fix $x_4=x_8=-1$, all other coordinates
with index at least four to $+1$, and leave $z=(x_1,x_2,x_3)$ free.
Rows $1,2,3,4$ have fixed contribution $-4$ and free normals
\[
 T=\{(-1,1,-1),(1,-1,-1),(-1,-1,1),(1,1,1)\}.
\]
They impose $t\cdot z\ge-1$ for $t\in T$, defining the parity
tetrahedron $\conv T$. Row $8$ duplicates row $4$ on this face;
all other rows have free normals in $T$, with fixed contribution
$8$ in row $0$, $4$ in rows $13,14,15$, and $0$ otherwise. Thus the cube-face section is
exactly this tetrahedron.

The degree-two point
\[
 h=(0,0,0,-2,2,2,2,-2,2,2,2,2,2,2,2)
\]
lies in $2P_{5,1}\cap(2+2\Z)^{15}$: $h/2$ is the face center.
Any decomposition into two source-lattice points forces both onto
the face, whose only odd-coordinate points are $T$. Their free
coordinates would have to be an antipodal pair, but $T$ has none.
Hence $h$ has no such decomposition.
\end{proof}

\section{Rounding lemmas and normality bands}

A sign vector has entries in $\{-1,1\}$.
The role of three constraints comes from a small parity phenomenon.
A nonintegral vertex of the slab intersection has exactly two
half-integer coordinates. Their four simultaneous roundings remain
in the box, and each positive-width signed slab excludes at most
one choice. Three slabs therefore leave a feasible rounding.
The next two lemmas make this argument precise, and the subsequent
objective estimate controls the error in every unprotected row.

\begin{lemma}\label{common:small-matrix}
An invertible sign matrix of order at most three has half-integral
inverse and maps common-parity integer right-hand sides to integer
solutions. For an invertible matrix of order three, every inverse row
has exactly two nonzero entries, each equal to $\pm1/2$.
\end{lemma}
\begin{proof}
At order one this is immediate. At order two the determinant is
$\pm2$ and the inverse entries are $\pm1/2$. At order three,
normalize the first row and column to ones and subtract the first row
from each other row. Direct evaluation gives determinant $\pm4$;
the cofactor formula gives the asserted inverse rows. In each inverse
row the numerator is a sum or difference of two right-hand sides,
which is even when their parities agree.
\end{proof}

\begin{lemma}[Three-slab feasibility]\label{common:slabs}
Let $B$ be a bounded box with integer endpoints and let $a_i$ be sign
vectors, $1\le i\le r,\ 0\le r\le3$. Suppose the integers $\ell_i$ have a
common parity and the integers $u_i$ have a common parity. If
\[
 R=B\cap\{y:\ell_i\le a_i\cdot y\le u_i,\ 1\le i\le r\}
\]
is nonempty, it contains an integer point. Every nonintegral vertex of $R$ has exactly two noninteger coordinates,
both in $\Z+\tfrac12$. If the lower
and upper endpoint parity classes agree, every vertex is integral.
\end{lemma}
\begin{proof}
For $r=0$, the intersection is the integer box itself, and the assertions are immediate.
For $r>0$, at a vertex at most three coordinates are strictly between their box
bounds: active slab rows must span the coordinates not fixed by active
box constraints. Select an independent set of the corresponding
restricted rows. Lemma~\ref{common:small-matrix} makes their solution
half-integral. An odd number of noninteger free coordinates would
give a noninteger value in each selected signed row equation, so a
nonintegral vertex has exactly two such coordinates. If all endpoints
share one parity, subtracting the fixed-coordinate contributions
preserves common parity, and the same lemma makes the vertex integral.

Otherwise every slab width is odd and, by nonemptiness, positive.
Round the two noninteger coordinates independently in the four
possible ways. Every resulting point is integral and remains in the
box. Each slab value was an integer and changes by $-1,0,0,1$.
A positive-width integer slab excludes at most one rounding. At most
three slabs cannot exclude all four. The argument also covers
lower-dimensional intersections, since an independent active subset
still has full rank in the free coordinates.
\end{proof}

Write $B(a)=\prod_j[\lfloor a_j\rfloor,\lceil a_j\rceil]$.

\begin{lemma}[Rounding estimate]\label{common:rounding}
Suppose a three-slab system as in Lemma~\ref{common:slabs}, restricted to
$B(a)$, contains $a\in\R^d$. There is a feasible integer $p$ with
$\|p-a\|_1\le t/2+1/4$, where $t$ is the number of noninteger coordinates of $a$.
\end{lemma}
\begin{proof}
Translate the $t$ nontrivial rounding intervals to $[0,1]$, with center
coordinates $f_i\in(0,1)$, and minimize
\[
 G(y)=\sum_i\bigl(f_i+(1-2f_i)y_i\bigr).
\]
At integer points this equals distance from the center.
For an integral optimal vertex, the error is at most
$G(f)=\sum_i2f_i(1-f_i)\le t/2$. Otherwise let $j,l$ be the two
noninteger coordinates of an optimal vertex, both in $\Z+\tfrac12$. Feasible rounding gives
\begin{equation}\label{common:weighted-cost}
 G(p)\le G(f)+\frac{|1-2f_j|+|1-2f_l|}{2}.
\end{equation}
For either selected coordinate, setting $u=|1-2f_i|$ gives
\[
 2f_i(1-f_i)+\frac{|1-2f_i|}{2}
 =\frac{1-u^2+u}{2}\le\frac58.
\]
Every other coordinate contributes at most $1/2$. Thus the total
is at most $(t-2)/2+2(5/8)=t/2+1/4$. This pointwise estimate is valid
for whichever two indices occur at the optimal vertex; no independence
assumption is used. At $t=0$ take $p=a$, and at $t=1$ the optimal
vertex is integral.
\end{proof}

\begin{theorem}[Direct extraction]\label{common:extraction}
Let $d\ge2$, let $n$ be a positive integer, and let
\[
 Q=\{y\in[0,n]^d:a_i\cdot y\ge b_i\ \forall i\},
\]
where the $a_i$ are sign vectors and the integer $b_i$ have a common
parity. Put $C=\lfloor d/2\rfloor$. If at every box point at most
three rows have slack below $C$, then $Q$ is integral and has IDP.
\end{theorem}
\begin{proof}
For $Z\in kQ\cap\Z^d$, $k\ge2$, put $\xi=Z/k$ and protect the at most
three rows whose slack at $\xi$ is below $C$. In $B(\xi)$ impose
\[
 b_i\le a_i\cdot p\le a_i\cdot Z-(k-1)b_i.
\]
These slabs contain $\xi$. The lower endpoints have common parity, as
do the upper endpoints because
$a_i\cdot Z\equiv\sum_jZ_j\pmod2$. Their integer rounding box lies
inside both factor boxes. Lemma~\ref{common:rounding} gives an integer
$p$ with row error at most $d/2+1/4<C+1$.
An omitted row starts with lower slack at least $C$ and upper slack
at least $(k-1)C$. Its final integer slacks are greater than $-1$,
hence nonnegative. Thus $p\in Q$ and $Z-p\in(k-1)Q$.
Induction gives IDP in every degree. A denominator multiple of a
rational vertex decomposes into integer points having that vertex
as average; extremality forces all summands to equal it. This also
proves integrality.
\end{proof}

\subsection{Quartet bounds}
\begin{lemma}\label{common:general-quartet}
For four distinct simplex rows in any marked Hadamard realization,
$\|v_1+v_2+v_3+v_4\|_1\le2(d-3)$.
\end{lemma}
\begin{proof}
Let $s_j$ be the sum of the four full rows in column $j$.
Orthogonality gives $\sum_js_j^2=4N$. The values are
$0,\pm2,\pm4$, so $|s_j|\le s_j^2/2$, with a saving of four
when $|s_j|=4$. The marked column has sum four. Its saving and its
deletion therefore give an upper bound $2N-8=2(d-3)$.
\end{proof}

\subsection{A general normality bound}
\begin{theorem}\label{syl:generic}
For every marked real Hadamard realization of order $N\ge8$, with
$d=N-1$, the truncation $P_{m,n}$ is integral and normal whenever
$m\ge(d-3)n/2+d-1$.
\end{theorem}
\begin{proof}
Normalize the marked column to ones. Choose one row and multiply
each remaining column by its entry in that row. This signed-coordinate
change preserves the cube, every source degree lattice, and addition.
The chosen simplex row now has sum $d$, and orthogonality makes every
other simplex row have sum $-1$. In the resulting integer chart, the right-hand sides
$b_i=(n\sum_jv_i(j)-m)/2$ have common parity: the two row sums
are $d$ and $-1$, so their difference is $nN/2$, an even integer.
Set $C=\lfloor d/2\rfloor$. Four slacks below $C$ would imply
$v_i\cdot x<-m+2C\le-(d-3)n/2$ for those four rows.
The general quartet estimate of Lemma~\ref{common:general-quartet},
$\|\sum_{i=1}^4v_i\|_1\le2(d-3)$, contradicts this in the cube.
Theorem~\ref{common:extraction} proves IDP.
The chart maps a degree-$r$ integer summand $p$ to
$2p-rn\ones$, so every split uses the full source lattice.
\end{proof}

The bound in Theorem~\ref{syl:generic} adds admissible parameters beyond
Corollary~\ref{common:upper-normality} exactly when $n\ge d$.

\subsection{Sylvester profiles and all-order bounds}
For the Sylvester matrix of order $N=2^k\ge16$, rows and columns are indexed
by $\F_2^k$, with entries $(-1)^{r\cdot j}$; the marked column is $j=0$.
Put $\rho=3N/8-1$.
A quadruple is \emph{closed} if the entrywise product of its full rows
is constant~\cite[Section~3]{Orrick}. For Sylvester rows this means their
labels sum to zero. Otherwise, translating one label to zero leaves
three independent labels; we call this the \emph{rank-three} case.
The following unmarked-column profiles follow by a character calculation.
\begin{lemma}[Sylvester quartets]\label{common:quartets}
For four distinct normalized rows, the absolute spatial coordinate
sums have one of the profiles
\[
 \begin{array}{c|ccc|c}
 &0&2&4&\text{total weight}\\ \hline
 \text{rank three}&3N/8&N/2&N/8-1&3N/2-4\\
 \text{closed}&3N/4&0&N/4-1&N-4 .
 \end{array}
\]
\end{lemma}
\begin{proof}
Multiply all four full character rows by the first. The remaining
labels are \(0,a,b,c\), with \(a,b,c\) distinct and nonzero.
If \(a+b+c=0\), their rank is two and the absolute sum over the
four equally repeated sign patterns is \(4,0,0,0\). Otherwise their
rank is three; over eight patterns the absolute sums are \(4\)
once, \(2\) four times and \(0\) three times. Deleting the trivial
column removes one weight four, giving the displayed profiles.
\end{proof}

\begin{theorem}[Sharp Sylvester integrality and all-order normality]
\label{syl:all-orders}
Let $N=2^k\ge16$ and $\rho=3N/8-1$. For every marked column of every matrix in the signed row/column
equivalence class of the Sylvester matrix, and positive odd source
parameters $n<m<(N-1)n$,
\[
 m\ge\rho n\ \Longrightarrow\ P_{m,n}\text{ is integral},\qquad
 m\ge\rho n+N-2\ \Longrightarrow\ P_{m,n}\text{ is normal}.
\]
For every positive odd $n$, the preceding admissible value
$m=\rho n-2$ has an explicit nonlattice vertex. Thus $\rho$ is the
sharp coefficient for a uniform upper integrality band; this does
not classify all smaller parameters.
\end{theorem}
\begin{proof}
First use the fixed character realization. Lemma~\ref{common:quartets}
gives $\|v_1+v_2+v_3+v_4\|_1\le3N/2-4=4\rho$.
At a vertex $x$, let $J$ be the coordinates strictly inside the cube.
The active simplex rows restricted to $J$ span $\R^J$. If $|J|\ge4$,
select four independent restrictions, and let $w$ be the sum of their
full rows. Activity gives
\[
 -4m=w\cdot x\ge-n\|w\|_1\ge-4\rho n.
\]
For $m>\rho n$ this is impossible. At equality, each nonzero $w_j$
forces $x_j=-n\operatorname{sign}(w_j)$, a cube endpoint. Hence $w$
vanishes on $J$, contradicting independence of the selected restrictions.
Thus $|J|\le3$, including at equality. In the integer chart, the row
sums $N-1$ and $-1$ give common-parity right-hand sides, as in
Theorem~\ref{syl:generic}; subtracting fixed coordinates preserves
that parity. Lemma~\ref{common:small-matrix} therefore makes
all free coordinates integral.

For normality put $C=\lfloor(N-1)/2\rfloor=(N-2)/2$.
At $m\ge\rho n+N-2$, four chart slacks below $C$ would imply
$v_i\cdot x<-m+2C\le-\rho n$ for those rows, contradicting the
quartet norm bound in the cube. Theorem~\ref{common:extraction}
and the degree-lattice chart prove normality.

For sharpness write $q=N/8$ and split a column label $j=(z,u)$
into three low bits $z$ and $k-3$ high bits $u$. Let $\epsilon(j)$
be the character of the first high bit and set
\[
 f(j)=\begin{cases}
 -1,&\wt(z)=0,1,\\
 \epsilon(j),&\wt(z)=2,\\
 1,&\wt(z)=3.
 \end{cases}
\]
Define the unnormalized Walsh transform by
$\widehat f(r)=\sum_{j\in\F_2^k}f(j)(-1)^{r\cdot j}$.
For a row label $(a,b)$, write $h=\wt(a)$. Character orthogonality
in the high bits gives
\[
\begin{array}{c|c|c}
 b&h&\widehat f(a,b)\\\hline
 0&0,1&-3q\\
 0&2,3&q\\
 \text{first high unit}&0,3&3q\\
 \text{first high unit}&1,2&-q\\
 \text{other}&\text{all}&0
\end{array}
\]
The cases $b=0$ and $b$ equal to the first high unit reduce respectively to the
low-bit sums $-4+2h+(-1)^h$ and
$\sum_{\wt(z)=2}(-1)^{a\cdot z}$, the latter equal to $3,-1,-1,3$
at weights $0,1,2,3$.

Delete column zero and start with $x_j=nf(j)$. Since $f(0)=-1$,
one has $v_r\cdot x=n(\widehat f(r)+1)$. Exactly rows $0,1,2,4$ have value $-\rho n$; every other row is at
least $-(q-1)n$. Add one at columns $1,2,4$ and subtract one at
column $7$. A row of low weight $h$ changes by
$3-2h-(-1)^h$, equal to $2$ for $h=0,1$ and $-2$ otherwise.
At $m=\rho n-2$ the four selected rows are tight, and all other
rows are feasible because their slack is at least $2qn-4\ge0$.
Extra equalities at $N=16,n=1$ cause no difficulty.
The only free coordinates are $-n+1,-n+1,-n+1,n-1$, and the
selected restrictions are
\[
 \begin{pmatrix}1&1&1&1\\-1&1&1&-1\\
 1&-1&1&-1\\1&1&-1&-1\end{pmatrix},
 \qquad MM^{\mathsf T}=4I.
\]
Together with the fixed cube coordinates these give full rank, so
$x$ is a vertex. Its free coordinates have the wrong source parity.
The parameters are admissible since $\rho\ge5$ is odd and
$n<\rho n-2<(N-1)n$. For any coefficient $c<\rho$, sufficiently
large odd $n$ gives $\rho n-2\ge cn$, proving uniform sharpness.

Finally, write any signed-equivalent Sylvester matrix as
$A_{ij}=\epsilon_i\sigma_j(-1)^{r_i\cdot c_j}$ with bijective labels.
Normalizing its marked column $j_0$ gives
\[
 A_{ij_0}A_{ij}=\sigma_{j_0}\sigma_j
                    (-1)^{r_i\cdot(c_j+c_{j_0})}.
\]
The unmarked translated labels are exactly the nonzero vectors.
Thus the marked simplex is a signed coordinate permutation of the
fixed one, preserving cubes, vertices, every degree's affine lattice,
and sums. All three assertions transfer.
\end{proof}

The normality band in Theorem~\ref{syl:all-orders} adds admissible parameters
beyond Corollary~\ref{common:upper-normality} exactly when $n\ge8-16/N$:
for odd $n$, this means $n\ge7$ at $N=16$ and $n\ge9$ at $N\ge32$.

\section{Scope and further questions}
The explicit dimension-eleven vertex answers the universal integrality
question negatively. The smoothness criterion settles the smoothness
question under the source assumptions and excludes this family as a
source of smooth nonnormal examples. The normality bounds address
specified upper parameter regions; the remaining parameters are
not classified by these results.

The Sylvester band $m\ge\rho n$, with $\rho=3N/8-1$, is sharp because
$m=\rho n-2$ has a nonlattice vertex at every odd scale. This rules out
any smaller uniform coefficient.
Normality is guaranteed within additive offset
$N-2$ at every Sylvester order $N\ge16$. The integral nonnormal
member $P_{5,1}$ shows that the integrality threshold itself need not
imply normality.

\paragraph{Machine-checked companion.}
The Lean~4 companion~\cite{Lean4,Mathlib} proves formal counterparts of
all 17 numbered results for the polytopes, affine lattices and facet
normals defined here, including every positive decomposition degree
and the marked-matrix transfers. It uses alternative proofs of
direct-extraction integrality (active-row parity and compact convexity)
and integrality in Proposition~\ref{common:nonnormal-example}
(Theorem~\ref{syl:all-orders}). The ancillary directory
\texttt{anc/lean} contains the sources, build instructions and the map
\texttt{paper\_coverage.json}. It uses Lean~4.34.0 and Mathlib commit
\texttt{5ed29652}, with the full revision pinned in the manifest.
This coverage concerns numbered statements; prose, attribution and
priority claims are reviewed separately. The proofs use only standard
Lean axioms, recorded in the companion's audit.

\paragraph{AI assistance.}
Generative AI tools assisted with proof exploration, verification code,
and editing. The author is responsible for the paper's entire content.

\end{document}